\documentclass{article}
\usepackage[utf8]{inputenc}
\usepackage{amsfonts,amsmath}
\usepackage{amsthm}
\usepackage{amscd}
\usepackage{blindtext}
\usepackage{multicol}
\usepackage{float}
\usepackage[bibstyle=authortitle,backend=biber, style=numeric]{biblatex}
\usepackage{hyperref}
\usepackage{geometry}
\usepackage{comment}
\usepackage{color}
\usepackage{esint}
\usepackage{graphicx}
\usepackage{xcolor}
\usepackage{tikz}
\usetikzlibrary{mindmap}
\graphicspath{ {./Brenier/} }
\theoremstyle{definition} \newtheorem{definition}{Definition}[section]
\theoremstyle{definition} \newtheorem{remark}[definition]{Remark}
\theoremstyle{plain} \newtheorem{lemma}[definition]{Lemma}
\theoremstyle{plain} \newtheorem{proposition}[definition]{Proposition}
\theoremstyle{plain} \newtheorem{theorem}[definition]{Theorem}
\theoremstyle{plain} 
\theoremstyle{definition} \newtheorem{example}[definition]{Example}
\theoremstyle{definition}

\DeclareMathOperator{\BV}{BV}

\newcommand{\R}{\mathbb{R}}

\newcommand{\N}{\mathbb{N}}

\newcommand{\T}{\mathbb{T}}
\newcommand{\TV}{\text{\rm Tot.Var.}}

\renewcommand{\L}{\mathscr L}

\renewcommand{\L}{\mathcal L}

\numberwithin{equation}{section}

\def\XXint#1#2#3{{\setbox0=\hbox{$#1{#2#3}{\int}$ }
		\vcenter{\hbox{$#2#3$ }}\kern-.6\wd0}}

\title{An example of a weakly mixing $\BV$ vector field which is not strongly mixing}
\author{Martina Zizza\footnote{S.I.S.S.A., via Bonomea 265, 34136 Trieste, Italy. E-mail: mzizza@sissa.it}}
\date{\today}

\begin{document}

\maketitle
\begin{abstract}
   
   \noindent We give an example of a weakly mixing vector field $b\in L^\infty([0,1],\BV(\T^2))$ which is not strongly mixing, in the setting first introduced in \cite{Bianchini_Zizza_residuality}. The example is based on a work of Chacon \cite{Weaklymixingnotstrongly} who constructed a weakly mixing automorphism which is not strongly mixing on $([0,1],\mathcal{B}([0,1]),|\cdot|)$, where $\mathcal{B}([0,1])$ are the Borel subsets of $[0,1]$ and $|\cdot|$ is the one-dimensional Lebesgue measure. 
\end{abstract}


\noindent \begin{center} 14/2022/MATE \end{center}
\tableofcontents
\section{Introduction}
We consider the measure space $(\T^2,\mathcal{B}(\T^2),|\cdot|)$ where the torus $\T^2$ is the unit square $[0,1]^2$ with periodic boundary conditions, $\mathcal{B}(\T^2)$ denotes the Borel subsets of $\T^2$ and $|\cdot|$ represents the normalized Lebesgue measure.  We consider the group of automorphisms on the torus, that is
\begin{equation*}
    G(\T^2)=\lbrace T:\T^2\rightarrow \T^2 \text{ invertible and measure-preserving}\rbrace,
\end{equation*}
made of all measurable and invertible maps that preserve the Lebesgue measure, in symbols:
\begin{equation*}
    |T^{-1}(A)|=|A|,\quad\forall A\in\mathcal{B}(\T^2).
\end{equation*}
Among automorphisms, we look for those that \emph{mix} sets, that is
\begin{itemize}
    \item $T$ is \textbf{\emph{weakly mixing}} if $\forall A,B\in\mathcal{B}(\mathbb{T}^2)$
        \begin{equation*}
            \lim_{n\to\infty}\frac{1}{n}\sum_{j=0}^{n-1}\left[|T^{-j}(A)\cap B|-|A||B|\right]^2=0;
        \end{equation*}
        \item $T$ is \textbf{\emph{(strongly) mixing}} if $\forall A,B\in\mathcal{B}(\mathbb{T}^2)$
        \begin{equation*}
            \lim_{n\to\infty}|T^{-n}(A)\cap B|=|A||B|.
        \end{equation*}
\end{itemize}
One can easily observe that the strong mixing implies the weak mixing, while there are counterexamples of weakly mixing automorphisms that are not strongly mixing (see for instance \cite{Weaklymixingnotstrongly}). \\

\noindent These objects are well understood from the point of view of Ergodic Theory, whereas their counterpart in the dynamics of incompressible fluids is a recent research topic that presents many challenging open questions. In Fluid Dynamics one focuses mostly on (divergence-free) vector fields $b=b(t,x)$ with $t\in[0,1]$, $x\in\T^2$, whose flow $X(t)$, that is the solution of the following ODE system
\begin{equation*}
    \begin{cases}
    \dot{X}(t,x)=b(t,X(t,x)),\\
    X(0,x)=x,
    \end{cases}
\end{equation*}
is an automorphism of $G(\T^2)$ when evaluated at time $t=1$. Indeed, thanks to the incompressibility condition ($\text{div } b=0$), at any instant of time the flow $X(t)$ can be regarded as an invertible and measure-preserving map from $\T^2$ into itself.
The theory of DiPerna and Lions \cite{DiPerna:Lions} and Ambrosio \cite{Ambrosio:BV} establishes existence,  uniqueness and stability for the (Regular Lagrangian) flow of divergence-free vector fields living in $L^\infty([0,1],W^{1,1}_{\text{loc}}(\T^2))$ or $L^\infty([0,1],\BV_\text{loc}(\T^2))$, giving the possibility to exploit on the one side the rough regularity in space of the vector field, on the other side to establish deep connections with Ergodic Theory. To be more precise,
\begin{definition}
%    \label{RLF}
    Let $b\in L^1([0,1]\times \T^2; \T^2)$. A map $X:[0,1]\times \T^2\rightarrow \T^2$ is a \emph{Regular Lagrangian Flow} (RLF) for the vector field $b$ if
    \begin{enumerate}
        \item for a.e. $x\in \T^2$ the map $t\rightarrow X(t,x)$ is an absolutely continuous integral solution of
        \begin{equation*}
            \begin{cases}
    \frac{d}{dt} x(t)=b(t, x(t)); & \\
    x(0)=x.
    \end{cases}
        \end{equation*}
        \item there exists a positive constant $C$ independent of $t$ such that
        \begin{equation*}
%            \label{comp:cond}
            |X(t)^{-1}(A)|\leq C|A|,\quad\forall A\in\mathcal{B}(\T^2).
        \end{equation*}
    \end{enumerate}
\end{definition}
\noindent The above discussion motivates the following
\begin{definition}
      Let $b\in L^\infty([0,1],\BV_\text{loc}(\T^2))$ be a divergence-free vector field. We say that $b$ is weakly mixing  (resp. strongly mixing) if its unique RLF $X(t)$, when evaluated at $t=1$, is a weakly mixing automorphism (resp. strongly mixing). 
  \end{definition}  
\noindent The definition of mixing vector fields was first given in \cite{Bianchini_Zizza_residuality}, but there are examples of strongly mixing vector fields in previous literature. For instance, in  \cite{univ:mixer} the authors give an explicit example of a strongly mixing vector field $b\in L^\infty([0,1],W^{s,p}(\T^d))$ for $s<\frac{1+\sqrt 5}{2}$ and $p\in[1,\frac{1}{2s+1-\sqrt 5})$ whose RLF at time $t=1$ is the \emph{Folded Baker's map}.
 We should remark that the advantage of the rough regularity of the vector field, especially in the case of the $\BV$ regularity, is that it allows for rigid \emph{cut and paste} motions, since the flows originated by these vector fields do not preserve the property of a set to be connected. These constructions would be hard to reproduce for vector fields with higher regularity in time and space. Nevertheless, a stochastic approach (see for example \cite{Almostsure}) investigates mixing vector fields with higher regularity, but does not furnish  deterministic examples of vector fields with the desired mixing properties.\\

\noindent The aim of this paper is the construction of a vector field $b\in L^\infty([0,1],\BV(\T^2))$, which is weakly mixing but not strongly mixing. As proved in \cite{Halmos:weak:mix} for automorphisms and in \cite{Bianchini_Zizza_residuality} for vector fields, the weakly mixing behaviour is \emph{typical}, while strongly mixing vector fields are \emph{few} in the sense of Baire Category Theorem (see also \cite{Alpern},\cite{Halmos:ergodic},\cite{Halmos:lectures},\cite{Oxtoby},\cite{Rokhlin},\cite{Weiss}). Nevertheless it is hard to give examples of weakly mixing automorphisms/vector fields that are not strongly mixing.
As written above, Zlato\v{s} and Elgindi construct a strongly mixing vector field in \cite{univ:mixer}, while there are no examples in literature (up to our knowledge) of vector fields that are weakly mixing but not strongly mixing. We want especially to underline that the parallelism with Ergodic Theory, which has investigated different stages of the mixing behaviour (see for instance \emph{mild mixing}, \emph{light mixing}, \emph{partial mixing}), can give new insights from the point of view of the dynamics of an incompressible fluid. \\

\noindent Our examples is based on a work of Chacon \cite{Weaklymixingnotstrongly}, who constructed a weakly mixing automorphism that is not strongly mixing, in the one-dimensional space $([0,1],\mathcal{B}([0,1]),|\cdot|)$, where $|\cdot|$ is the one-dimensional Lebesgue measure. The importance of his work is that a general procedure to build up weakly mixing automorphisms that are not strongly mixing is given. Indeed his example easily extends to multiple dimensions, but we will focus on the dimension $d=2$ to avoid technicalities and to provide some \emph{visual} expressions of the vector fields under consideration. In particular we will construct a divergence-free vector field $b\in L^\infty([0,1],\BV(\T^2))$ whose Regular Lagrangian Flow $X(t)$, when evaluated at time $t=1$, is a weakly mixing automorphism that is not strongly mixing. \\

\noindent The idea of the paper is the definition of a setting (see \emph{Configurations} in Section \ref{subS_rows_columns}) that helps to relate an automorphism, call it $U$, to the vector field $b^U$ whose RLF at time $t=1$ is $U$. \\

\noindent A \emph{configuration} $\gamma$ is a $n\times n$ matrix that takes values in the set $\lbrace 1,2,\dots,n^2\rbrace $. Moreover we ask that $\gamma_{ij}=\gamma_{hk}$ iff $i=h,j=k$. Every configuration indeed represents an enumeration of the subsquares of the torus of the grid $\N\times\N\frac{1}{n}$ (see Figures \ref{fig:torus:tass},\ref{fig:torus:tass:2}). A \emph{movement} $T:\mathcal{C}(n)\rightarrow\mathcal{C}(n)$ is a one-to-one map from the space of configurations $\mathcal{C}(n)$ into itself, that is, it is a \emph{permutations} of the subsquares of the grid $\N\times\N\frac{1}{n}$. Among movements the following are relevant for Chacon's example: simple exchange, sort, rotation (Section \ref{subS_rows_columns}), and each one of them can be realized as the flow, evaluated at time $t=1$, of some $\BV$ divergence-free vector field (see Subsection \ref{Ss:vf} and Figure \ref{fig:torus:tass:3}). \\

\begin{figure}
    \centering
\includegraphics[scale=0.5]{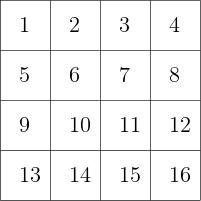}
    \caption{The grid on the torus with $k=4$ and an enumeration of subsquares.}
    \label{fig:torus:tass}
\end{figure}
\begin{figure}
    \centering
\includegraphics[scale=0.5]{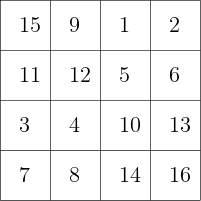}
    \caption{The automorphism $U_4$ sends the starting configuration into this final one.}
    \label{fig:torus:tass:2}
\end{figure}
\begin{figure}
    \centering
\includegraphics[scale=0.5]{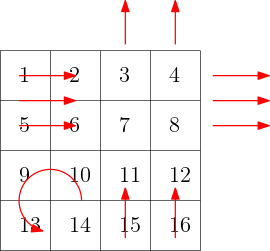}
    \caption{Vector fields on the torus that move rigidly the subsquares.}
    \label{fig:torus:tass:3}
\end{figure}
\noindent 
The key example that one has to keep in mind while approaching this problem is the \emph{15 puzzle} where one performs rigid movements on subsquares in order to reach the desired configuration (see Figure \ref{fig:15puzzle}). \\

\begin{figure}
    \centering \includegraphics[scale=0.5]{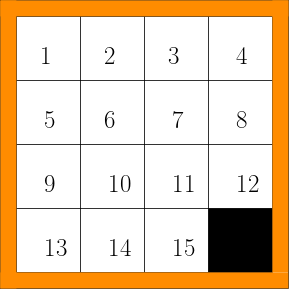}
    \caption{The 15 puzzle.}
    \label{fig:15puzzle}
\end{figure}

\noindent Chacon's method provides a sequence of automorphisms $\lbrace U_k\rbrace$ that are permutations of the subsquares of the grid $\N\times\N\frac{1}{k}$. To give an intuition, we fix for example $k=4$ and we assume to enumerate squares as in Figure \ref{fig:torus:tass}. Chacon's automorphism $U_4\in G(\T^2)$  moves the previous enumeration of squares (the starting configuration) into the final configuration in Figure \ref{fig:torus:tass:2}. Then he considers the limit map $U=\lim_{k\to\infty}U_k$, which is a weakly mixing automorphism that is not strongly mixing (see Sections \ref{Ss:chacon}, \ref{S:twod}). \\

\noindent Using Chacon's construction and the setting of configurations, we prove in particular the following result (Section \ref{S:wmvf}):
\begin{theorem}
There exists a divergence-free vector field $b^U\in L^\infty_t\BV_x$ whose RLF $X^U(t)$ when evaluated at time $t=1$ is the Chacon's map $U$, that is $b^U$ is a weakly mixing vector field that is not strongly mixing.
\end{theorem}
\noindent We conclude remarking that all these constructions are possible assuming to work in $[0,1]^d$ instead of $\mathbb{T}^d$. \\

\noindent \textbf{Plan of the paper.} After a short introduction on automorphisms in Ergodic Theory and on Regular Lagrangian Flows (Subsection \ref{Ss:RLF}), we introduce the Canonical Chacon's transformation in the one-dimensional space (Section \ref{Ss:chacon}), that will be a model for the extension to higher dimensions (see Section \ref{S:twod}). In Section \ref{subS_rows_columns} we introduce the notions of \emph{configurations} and \emph{movements}, and in Section \ref{Ss:vf} we describe how to generate vector fields acting on the torus as the movements previously defined, giving some estimates on their total variation. In Section \ref{S:twod} we describe the two dimensional automorphism $U$ in terms of movements, and finally in Section \ref{S:wmvf} we construct the RLF flow $X^U(t)$ with the property that $X^U(1)=U$.
\\

\noindent \textbf{Aknowledgments.} The author thanks Gianluca Crippa for having proposed the problem and Stefano Bianchini for many useful comments.

\section{Preliminaries}
For notations and preliminaries we refer mainly to \cite{Bianchini_Zizza_residuality}. Throughout the paper we will consider divergence-free vector fields $b:[0,1]\times \mathbb{T}^2\rightarrow \R^2$ in the space $L^\infty([0,1],\BV(\mathbb{T}^2))$ (or alternatively in $L^\infty([0,1],\BV(\R^2))$ with $\text{supp } b\subset\subset[0,1]^2$), in short $b\in L^\infty_t\BV_x$. The measure space under interest is $(\mathbb{T}^2,\mathcal{B}(\mathbb{T}^2),|\cdot|)$  where $|\cdot|$ is the two dimensional Lebesgue measure, and  $\mathbb{T}^2$ is the unit square $[0,1]^2$ with periodic boundary conditions. We remark that the analysis of the next sections can be easily extended to higher dimensions, but we will focus on the two dimensional case to avoid technical issues.
%%%%%%%%%%%%%%%%%%%%%%%%%%%%%%%%%%%%%%%%%%%%%%%%%%%%%%%%
\subsection{Mixing automorphisms}
We call $G(\mathbb{T}^2)$  the group (w.r.t. the composition) of automorphisms of the torus, namely:
\begin{definition}
    An \emph{automorphism} of the measure space $(\mathbb{T}^2,\mathcal{B}(\mathbb{T}^2),|\cdot|)$ is an invertible map $T:\mathbb{T}^2\rightarrow\mathbb{T}^2$ measurable and measure-preserving, that is
    \begin{equation*}
    %\label{measure-preserving}
        |A|=|T(A)|=|T^{-1}(A)|, \quad \forall A\in\mathcal{B}(\mathbb{T}^2).
    \end{equation*}
\end{definition}
\noindent From now on we will say that $T$ is an automorphism of the torus meaning that it is an automorphism of the measure space $(\mathbb{T}^2,\mathcal{B}(\mathbb{T}^2),|\cdot|)$. We remark that the definitions and propositions that we give in this section hold true for $(\Omega,\Sigma,\mu)$ where $\Omega$ is locally compact and separable, while $\mu$ is a complete and normalized measure on $\Omega$.
\begin{definition} Let $T\in G(\mathbb{T}^2)$ be an automorphism. Then
\begin{itemize}
\item $T$ is \textbf{ergodic} if $\forall A\in\mathcal{B}(\T^2)$ such that $T^{-1}(A)=A$ then $|A|=0$ or $1$;
        \item $T$ is \textbf{\emph{weakly mixing}} if $\forall A,B\in\mathcal{B}(\mathbb{T}^2)$
        \begin{equation}
        \label{weak:mix}
            \lim_{n\to\infty}\frac{1}{n}\sum_{j=0}^{n-1}\left[|T^{-j}(A)\cap B|-|A||B|\right]^2=0;
        \end{equation}
        \item $T$ is \textbf{\emph{(strongly) mixing}} if $\forall A,B\in\mathcal{B}(\mathbb{T}^2)$
        \begin{equation}
        \label{mix}
            \lim_{n\to\infty}|T^{-n}(A)\cap B|=|A||B|.
        \end{equation}
    \end{itemize}
    \end{definition}
\noindent
It is easy to see that strong mixing implies weak mixing which implies the ergodicity. Let now $T\in G(\mathbb{T}^2)$, then the Kolmogorov Koopman operator (see Chapter 1 in \cite{Ergodic:theory})  $U_T:L^2(\mathbb{T}^2)\rightarrow L^2(\mathbb{T}^2)$ is defined as
\begin{equation}\label{koopman_operator}
    U_T f(x)\doteq f(T(x)), \quad\forall f\in L^2(\mathbb{T}^2).
\end{equation}
    We observe that any operator $U_T$ of the form considered has eigenfunctions $f=\text{const}$ corresponding to the eigenvalue $1$. Then we have the following:
    \begin{theorem}[Mixing, Theorem 2 \cite{Ergodic:theory}]\label{mixing theorem}
    $T$ is weakly mixing iff $U_T$ has no eigenfunctions which are not constants.
    \end{theorem}

\subsection{Mixing vector fields}\label{Ss:RLF}

We are interested in divergence-free vector fields $b$  whose flow $t\rightarrow X(t)\in C([0,1];G(\mathbb{T}^2))$ when evaluated at time $t=1$ is a weakly/strongly mixing automorphism.  When the velocity field $b:[0,1]\times\mathbb{T}^2\rightarrow\T^2$ is Lipschitz in space, uniformly in time, then its \emph{flow} is well-defined in the classical sense, indeed it is the map $X:[0,1]\times \mathbb{T}^2\rightarrow \mathbb{T}^2$ satisfying 
\begin{equation*}
 % \label{ODE}
    \begin{cases}
    \frac{d}{dt} X(t,x)=b(t, X(t,x)), & \\
    X_0(x)=x.
    \end{cases}
\end{equation*}
\noindent In the presence of discontinuities of the vector field (as the $\BV$ regularity in space) we can still give a notion of a flow known as the \emph{Regular Lagrangian Flow} associated with the vector field $b$.
\begin{definition}[Regular Lagrangian Flow]
%    \label{RLF}
    Let $b\in L^1([0,1]\times \T^2; \T^2)$. A map $X:[0,1]\times \T^2\rightarrow \T^2$ is a \emph{Regular Lagrangian Flow} (RLF) for the vector field $b$ if
    \begin{enumerate}
        \item for a.e. $x\in \T^2$ the map $t\rightarrow X(t)(x)$ is an absolutely continuous integral solution of
        \begin{equation*}
            \begin{cases}
    \frac{d}{dt} x(t)=b(t, x(t)), & \\
    x(0)=x.
    \end{cases}
        \end{equation*}
        \item there exists a positive constant $C$ independent of $t$ such that
        \begin{equation*}
%            \label{comp:cond}
            |X(t)^{-1}(A)|\leq C|A|,\quad\forall A\in\mathcal{B}(\T^2).
        \end{equation*}
    \end{enumerate}
\end{definition}
\noindent Existence, uniqueness and stability for Regular Lagrangian Flows were established in \cite{DiPerna:Lions} for the Sobolev case, and in \cite{Ambrosio:BV} for the $\BV$ regularity of the vector field. In particular, if $b\in L^\infty_t\BV_x$ is a divergence-free vector field (that is, the distributional divergence $D\cdot b(t)=0$ $\mathcal{L}^1$-a.e. $t$), then its unique RLF $t\rightarrow X(t)$ is a flow of automorphisms.
	\begin{definition}
  \label{def:mix:vect:fields}
      Let $b\in L^\infty([0,1],\BV_\text{loc}(\mathbb{T}^2))$ be a divergence-free vector field. We say that $b$ is weakly mixing (strongly mixing) if its unique RLF $X(t)$ when evaluated at $t=1$ is a weakly mixing (respectively strongly mixing) automorphism of $\mathbb{T}^2$. 
  \end{definition}
\noindent The above definition is motivated by the discussion in \cite{Bianchini_Zizza_residuality} (see page 13, Lemma 3.7). The main point of this work, as in \cite{Bianchini_Zizza_residuality},\cite{univ:mixer}, is the construction of a vector field whose RLF, when evaluated at time $t=1$, is an invertible and measure-preserving map which is weakly mixing but not strongly mixing. 
\noindent In the sequel we will build flows of measure-preserving maps originating from divergence-free vector fields; more precisely, if a flow $X:[0,1]\times \T^2\rightarrow \T^2$ is invertible, measure-preserving for $\mathcal{L}^1$-a.e. $t$ and the map $t\rightarrow X(t)$ is differentiable for $ \mathcal{L}^1$-a.e. $t$ and  $\dot{X}(t)\in L^1(\T^2)$, then the \emph{vector field associated with $X(t)$} is the divergence-free vector field defined by
	\begin{equation}
	\label{vect:field:ass}
	b(t)(x)=b(t,x)=\dot{X}(t)(X^{-1}(t,x)).
	\end{equation}	

\subsection{Canonical Chacon's Transformation}\label{Ss:chacon}
We present here the one dimensional canonical Chacon's Transformation \cite{Weaklymixingnotstrongly}, that can be easily extended to higher dimension (see for example Section \ref{S:twod} for the two dimensional): indeed, this construction is based on a general geometric approach which consists in mapping subintervals of the same length linearly onto each other (see  also \cite{Chen2015THENO} for further reference). Let us consider $I=[0,1]$ and let $|\cdot|$ be the Lebesgue measure. The aim is to construct a weakly mixing automorphism $T\in G(I)$ which is not strongly mixing. 
\begin{definition}
    A \emph{column} $C$ is a finite sequence of disjoint subintervals $J\subset I$ called \emph{levels}. The number of levels in a column is its \emph{height} $h$.
    \end{definition}
    
\noindent
%\begin{definition}
%label{chacon limit} the family $\lbrace T_k\rbrace $ is a Chacon limit family (boh una cosa del genere).
%\end{definition}
\noindent We define a family of automorphisms $\lbrace T_k\rbrace_k\subset G(I)$ by induction. 
Let $C_0$ be the column $C_0\doteq I_{0,1}=\left[0,\frac{2}{3}\right)$, and let the \emph{remaining set} be $R_0=\left[\frac{2}{3},1\right]$. The height of the column $C_0$ is $h_0=1$, since it has a unique level $I_{0,1}$. We divide $C_0$ into three disjoint subintervals with same length: $I_{0,1}(1)=\left[0,\frac{2}{9}\right)$, $I_{0,1}(2)=\left[\frac{2}{9},\frac{4}{9}\right)$, $I_{0,1}(3)=\left[\frac{4}{9},\frac{2}{3}\right)$. We call \emph{spacer} the interval $S_0=\left[\frac{2}{3},\frac{8}{9}\right)$ and $R_1=R_0\setminus S_0$. We observe that the spacer has the same length of $I_{0,1}(j)$ for $j=1,2,3$. We put the spacer on the top of the middle interval $I_{0,1}(2)$ (see Figure \ref{fig:actionT1}). We define the piecewise linear map $T_1:I\rightarrow I$ in the following way: 
\begin{equation*}
\begin{cases}
     T_1(I_{0,1}(1))=I_{0,1}(2), \\
     T_1(I_{0,1}(2))=S_0,    \\
     T_1(S_0)=I_{0,1}(3).
     \end{cases}
\end{equation*}
\begin{figure}[H]
    \centering
    \includegraphics[scale=0.5]{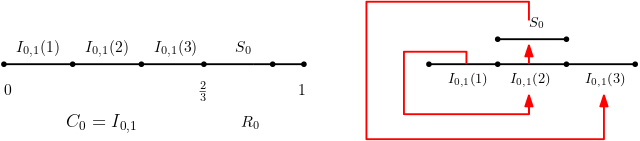}
    \caption{In the left figure the Column $C_0$, in the right figure the geometric representation of the action of the automorphism $T_1$.}
    \label{fig:actionT1}
\end{figure}
\noindent In the set $I\setminus(I_{0,1}(1)\cup I_{0,1}(2)\cup S_0)$ the map $T_1$ is defined in such a way that $T_1$ is invertible and measure-preserving (for simplicity we can assume $T_1(I_{0,1}(3))=I_{0,1}(1)$ and $T_1=id$ otherwise).
A useful notation to represent $T_1$ in a simpler way is using the language of permutations, that is
\begin{equation}
    T_1\llcorner_{[0,1]\setminus R_1}=\left(\begin{matrix}I_{0,1}(1) & I_{0,1}(2) & I_{0,1}(3) & S_0 \\ I_{0,1}(2) & S_0 & I_{0,1}(1) & I_{0,1}(3)\end{matrix}\right).
\end{equation}
\noindent We construct the column $C_1$ of height $h_1=3h_0+1=4$ putting one on the top of the other the intervals in the order $I_{0,1}(1),I_{0,1}(2),S_0,I_{0,1}(3)$: the intervals are arranged so that each point is located below its image (see Figure \ref{fig:columnC1}). We rename the levels $I_{1,1}=I_{0,1}(1), I_{1,2}=I_{0,1}(2), I_{1,3}=S_0, I_{1,4}=I_{0,1}(3)$.
\begin{figure}[H]
    \centering
    \includegraphics[scale=0.8]{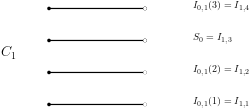}
    \caption{The column $C_1$ in which the levels are arranged one on top the other.}
    \label{fig:columnC1}
\end{figure}

\noindent
The inductive step is performed in the following way: we start with a column $C_n$ made of $h_n=3h_{n-1}+1$ levels  $I_{n,1},I_{n,2},\dots, I_{n,h_n}$. We divide each level into three disjoint consecutive subintervals of the same length, that is $I_{n,j}=I_{n,j}(1)\cup I_{n,j}(2)\cup I_{n,j}(3)$ for $j=1,\dots,h_n$. We consider the set $R_n=\left[\frac{3^{n+1}-1}{3^{n+1}},1\right]$ and we construct $S_n=\left[\frac{3^{n+1}-1}{3^{n+1}},\frac{3^{n+2}-1}{3^{n+2}}\right]$, with $|S_n|=\frac{2}{3^{n+2}}=|I_{n,j}(i)|$ for $j=1,\dots,h_n$, $i\in\lbrace 1,2,3\rbrace$ and we put it on the top of the middle interval $I_{n,h_n}(2)$ of the bottom level $I_{n,h_n}$ (see Figure \ref{fig:columnCn}). We finally define the piecewise linear map $T_{n+1}$ in the following way:
\begin{equation*}
\begin{cases}
     T_{n+1}(I_{n,j}(i))=I_{n,j+1}(i) \quad\forall j=1,\dots h_n-1, \quad\forall i\in\lbrace{1,2,3}\rbrace, \\
     T_{n+1}(I_{n,h_n}(1))=I_{n,1}(2),    \\
     T_{n+1}(I_{n,h_n}(2))=S_n, \\
     T_{n+1}(I_{n,h_n}(3))=I_{n,1}(1), \\
     T_{n+1}(S_n)=I_{n,1}(3), \\
      T_{n+1}=\text{id}\quad\text{otherwise}.
     \end{cases}
\end{equation*}
\begin{figure}
    \centering
    \includegraphics[scale=0.7]{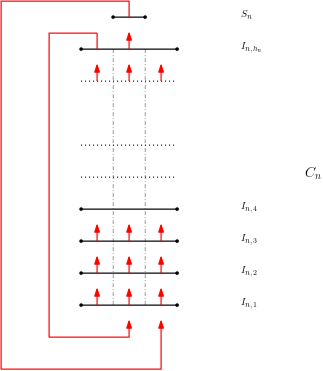}
    \caption{A column $C_n$ and the action of the automorphism $T_{n+1}$.}
    \label{fig:columnCn}
\end{figure}

\noindent The new column $C_{n+1}$ of height $h_{n+1}=3h_n+1$ is obtained arranging the levels starting from the bottom $I_{n,1}(1),I_{n,2}(1),\dots I_{n,h_n}(1), I_{n,1}(2),I_{n,2}(2),\dots I_{n,h_n}(2), S_n, I_{n,1}(3),\dots, I_{n,h_n}(3)$. We rename them as $I_{n+1,j}$ $j=1,\dots, h_{n+1}$ following the same order of the levels and we define $R_{n+1}=R_n\setminus S_{n}$.
\begin{definition}
The Canonical Chacon's map is the automorphism $T=\lim_{n\to\infty} T_n$.
\end{definition}

\noindent
Observe that the limit map $T$ is well defined, invertible and measure-preserving, indeed one can easily check that $T_{n+1}=T_n$ on $\left(\cup_{j=1}^{h_n-1} I_{n,j}\right)^c$.
\begin{theorem}[Chacon \cite{Weaklymixingnotstrongly}]\label{chacon wm not strongly}
$T$ is weakly mixing but not strongly mixing. 
\end{theorem}
\noindent We will give a proof of this theorem in the case of two dimensional examples (see Proposition \ref{prop:wm:not:sm}).

\section{Configurations and movements}\label{subS_rows_columns}
In this subsection we introduce the notion of rows, columns and configurations. The advantage of this abstract description is to simplify the action of two-dimensional automorphisms on the torus as the composition of a finite number of \emph{movements}, where each one can be described as the flow (evaluated at time $t=1$) of a divergence-free $\BV$ vector field. 
\begin{definition}

\label{row}
A row $r(i)$ of \emph{length} $n\in\N$ and index $i\in\N$ is a integer-valued $1\times n$ matrix $$r(i)=\begin{matrix}(i,j_1 & i,j_2 & \dots  & i,j_n)\end{matrix},$$ where $j_1,j_2,\dots,j_n\in\N$.
\end{definition}
\begin{definition}
\label{column}
A column $c(j)$ of \emph{length} $n\in\N$ and index $j\in\N$ is a integer-valued $n\times 1$ matrix $$c(j)=\left(\begin{matrix}i_1,j & i_2,j 
& \dots  & i_n,j\end{matrix}\right)^T,$$ where $i_1,i_2,\dots,i_n\in\N$.
\end{definition}

\begin{definition}
A configuration $\gamma$ of \emph{size} $n\in\N$ is a integer-valued $n\times n$ matrix, where the entries $\gamma_{i,j}\in\lbrace{1,2,\dots,n^2\rbrace}$ and $\gamma_{i,j}=\gamma_{h,k}$ iff $i=h,j=k$. We can denote $\gamma$ both \emph{by rows} as $\gamma=(r(1), r(2),\dots ,r(n))^T$ and \emph{by columns} $\gamma=(c(1),c(2),\dots,c(n))$. We call $\mathcal{C}(n)$ the \emph{space of configurations} of size $n$. \end{definition}
%\begin{definition} A \emph{subconfiguration} of $\gamma$ is $\gamma_{(i_1,\dots,i_h),(j_1,\dots,j_k)}$ a $h\times k$ submatrix of $\gamma$ defined by $\gamma_{(i_1,\dots,i_h),(j_1,\dots,j_k)}=(r(i_1),\dots,r(i_h))^T=(c(j_1),\dots, c(j_{k}))$, where we denote by $r(i_{h'})$ with abuse of notation the following $1\times k$ matrix $r(i_h')=(\begin{matrix} i_{h'} j_1 & i_{h'} j_2 &\dots & i_{h'} j_{k}\end{matrix})$ and we denote by $c(j_{k'})$ with abuse of notation the following $h\times 1$ matrix $c(j_{k'})=(\begin{matrix} i_1 j_{k'} & i_2 j_{k'} &\dots & i_h j_{k'}\end{matrix})^T$ .
%\end{definition}
\begin{remark}
One can easily observe that $\sharp\mathcal{C}(n)=n^2!$, being $\gamma$ all possible permutations of $\lbrace 1,2,\dots,n^2\rbrace$.
\end{remark}

\begin{definition}\label{adjacent}
Let $\gamma\in\mathcal{C}(n)$ be a configuration. Two entries  $\gamma_{i,j}$, $\gamma_{i',j'}$ are \emph{adjacent} if $|i'-i|=1$ and $j'=j$ or $i'=i$ and $|j'-j|=1$. 
\end{definition}

\noindent

\begin{definition}
A movement is a bijective map $S:\mathcal{C}(n)\rightarrow \mathcal{C}(n)$.
\end{definition}
\noindent
%\begin{definition}
%Let $\gamma\in\mathcal{C}(n)$ as above. Then we denote by $\rho_r(i),\rho_c(j):\mathcal{C}(n)\rightarrow\mathcal{C}(n)$ the following functions \begin{align*}
    %& \rho_r(i)(\gamma)=(r(1),r(2),\dots,r(i-1),\bar{r}(i),r(i+1),\dots,r(n))^T, \\ & \rho_c(j)(\gamma)=(c(1),c(2),\dots,c(j-1),\bar{c}(j),c(j+1),\dots,c(n)).\end{align*} 
%\end{definition}
\begin{definition}[Simple exchange]
Let $\gamma\in\mathcal{C}(n)$ and let $\gamma_{i,j}$, $\gamma_{i',j'}$ be two adjacent entries. For simplicity we can assume $i'=i+1, j'=j$ according to Definition \ref{adjacent}.  A \emph{simple exchange} is a map $E_s(i,j;i+1;j):\mathcal{C}(n)\rightarrow\mathcal{C}(n)$ that exchanges the two entries, that is
\begin{align*}
    E_s(i,j;i+1;j)\left(\begin{matrix} 
    \dots & \dots & \dots & \dots & \dots\\
    i,1 &\dots & i,j & \dots & i,n \\
    i+1,1 &\dots & i+1,j & \dots & i+1,n \\
    \dots & \dots & \dots & \dots & \dots
   \end{matrix}\right)=\left(\begin{matrix} 
    \dots & \dots & \dots & \dots & \dots\\
    i,1 &\dots & i+1,j & \dots & i,n \\
    i+1,1 &\dots & i,j & \dots & i+1,n \\
    \dots & \dots & \dots & \dots & \dots
    \end{matrix}\right)
\end{align*}
\end{definition}

\begin{definition}[Sort]
Let $\gamma\in\mathcal{C}(n)$ and let $i\in\N$ be some fixed row index. Let $j,j'\in \N$ with $j<j'$. Then the \emph{sort} operation on columns is $S_c(i;j,j'):\mathcal{C}(n)\rightarrow\mathcal{C}(n)$ defined by
\begin{align*}
    S_c(i;j,j')&\left(\begin{matrix}\dots &\dots &\dots &\dots &\dots &\dots &\dots \\
    \dots &i,j &i,j+1 &\dots & i,j'-1 & i,j' &\dots \\
    \dots &\dots &\dots &\dots &\dots &\dots &\dots
    \end{matrix}\right)=\\&\left(\begin{matrix}\dots& \dots &\dots &\dots &\dots &\dots &\dots \\
    \dots &i,j' &i,j & i,j+1 &\dots & i,j'-1 &\dots \\
    \dots &\dots &\dots &\dots &\dots &\dots &\dots
    \end{matrix}\right).
\end{align*}
Similarly, if $i,i',j\in\N$ and $i<i'$
$S_r(i,i';j):\mathcal{C}(n)\rightarrow\mathcal{C}(n)$ is defined by
\begin{equation}
    S_r(i,i';j)\left(\begin{matrix}\dots &\dots &\dots \\
    \dots &i,j &\dots \\
    \dots &i+1,j &\dots  \\
    \dots &\dots &\dots \\
    \dots &i',j &\dots \\
    \dots &\dots &\dots \\
    \end{matrix}\right)=\left(\begin{matrix}\dots &\dots &\dots \\
    \dots &i',j &\dots \\
    \dots &i,j &\dots  \\
    \dots &\dots &\dots \\
    \dots &i'-1,j &\dots \\
    \dots &\dots &\dots \\
    \end{matrix}\right).
\end{equation}
\end{definition}
\begin{example}\label{ex:sort}
We give the example of some configuration $\gamma$ and the new configuration $S_c(1;2,4)(\gamma)$ obtained by exchanging rows:\begin{equation*}
    \gamma=\left(\begin{matrix}1 & 2 & 3 & 4 \\ 5 & 6 & 7 & 8 \\ 9 & 10 & 11 & 12 \\ 13 & 14 & 15 & 16
    \end{matrix}\right) \qquad
    S_c(1;2,4)(\gamma)=\left(\begin{matrix}1 & 4 & 2 & 3 \\ 5 & 6 & 7 & 8 \\ 9 & 10 & 11 & 12 \\ 13 & 14 & 15 & 16
    \end{matrix}\right).
    \end{equation*}
\end{example}
\begin{remark}\label{rk:shift:vs:sort}
 We observe that if $j=1$ and $j'=n$ then the sort operation is simply a shift. Since we will work on the torus, the shifts are easier to perform than sort movements. In particular we will see that the cost of a shift, in terms of the total variation, is lower than the cost of a sort movement. 
\end{remark}
%\begin{definition}
%Let $\gamma=(r(1),r(2),\dots,r(n))^T=(c(1),c(2),\dots,c(n))$ be a configuration. A \emph{shift} of rows $s_r:\mathcal{C}(n)\rightarrow\mathcal{C}(n)$ is the following map
%\begin{equation}
  %  s_r(\gamma)=(r(n),r(1),r(2),\dots,r(n-1))^T.
%\end{equation}
%A \emph{shift} of columns $s_c:\mathcal{C}\rightarrow\mathcal{C}$ is the following map
%\begin{equation}
   % s_c(\gamma)=(c(n),c(1),c(2),\dots,c(n-1)).
    %\end{equation}
%\end{definition}

%\begin{definition}[Shift]
%Let $\gamma\in\mathcal{C}(n)$ as above. The shift $s_c(i):\mathcal{C}(n)\rightarrow \mathcal{C}(n)$ is the map that shift the columns of  the row $r(i)$, that is
%\begin{equation}
%    s_c(i)\left(\begin{matrix} 1,1 & 1,2 & \dots & 1,n \\
 %   \dots & \dots & \dots & \dots \\
 %   i,1 & i,2 & \dots & i,n \\
 %   \dots & \dots & \dots & \dots \\
  %  n,1 & n,2 & \dots & n,n 
  %  \end{matrix}\right)=\left(\begin{matrix} 1,1 & 1,2 & \dots & 1,n \\
  %  \dots & \dots & \dots & \dots \\
  %  i,n & i,1 & \dots & i,n-1 \\
  %  \dots & \dots & \dots & \dots \\
   % n,1 & n,2 & \dots & n,n 
  %  \end{matrix}\right).
%\end{equation}
%\end{definition}
\begin{definition}[Rotation]
Let $\gamma\in\mathcal{C}(n)$ and let $i\geq 2$ and $j<n$. Then the counterclockwise rotation $ R^{-}_{i,j}:\mathcal{C}(n)\rightarrow\mathcal{C}(n)$ is the following map
\begin{equation}
    R^{-}_{i,j}\left(\begin{matrix} 
          \dots & \dots & \dots & \dots\\
         \dots & i-1,j & i-1,j+1 &\dots &   \\
    \dots & i,j & i,j+1 & \dots \\
    
    \dots & \dots & \dots & \dots \\
    
    \end{matrix}\right)=\left(\begin{matrix} 
         \dots & \dots & \dots & \dots\\
           \dots & i-1,j+1 & i,j+1 &\dots &   \\
     \dots & i-1,j & i,j & \dots \\
    
    \dots & \dots & \dots & \dots \\
    \end{matrix}\right),
    \end{equation}
while the clockwise rotation $ R^{+}_{i,j}:\mathcal{C}(n)\rightarrow\mathcal{C}(n)$ is the following map
\begin{equation}
    R^{+}_{i,j}\left(\begin{matrix} 
          \dots & \dots & \dots & \dots\\
         \dots & i-1,j & i-1,j+1 &\dots &   \\
    \dots & i,j & i,j+1 & \dots \\
    
    \dots & \dots & \dots & \dots \\
    
    \end{matrix}\right)=\left(\begin{matrix} 
          \dots & \dots & \dots & \dots\\
         \dots & i,j & i-1,j &\dots &   \\
    \dots & i,j+1 & i-1,j+1 & \dots \\
    
    \dots & \dots & \dots & \dots \\
    \end{matrix}\right).
    \end{equation}
\end{definition}
\begin{example}
We give the example of some configuration $\gamma$ and the new configuration $ R^{-}_{2,3}(\gamma)$ obtained by a counterclockwise rotation:\begin{equation*}
    \gamma=\left(\begin{matrix}1 & 2 & 3 & 4 \\ 5 & 6 & 7 & 8 \\ 9 & 10 & 11 & 12 \\ 13 & 14 & 15 & 16
    \end{matrix}\right) \qquad
    R^{-}_{2,3}(\gamma)=\left(\begin{matrix}1 & 2 & 4 & 8 \\ 5 & 6 & 3 & 7 \\ 9 & 10 & 11 & 12 \\ 13 & 14 & 15 & 16
    \end{matrix}\right).
    \end{equation*}
\end{example}
\subsection{Flows of vector fields associated with movements}\label{Ss:vf}
From now on we will consider the two dimensional torus $\mathbb{T}^2= [0,1]^2$ with periodic boundary conditions. The idea of this subsection is to apply the previous description via configurations, rows and columns to the two dimensional torus. In particular our aim is to find vector fields whose flow acts on the torus as the movements previously defined (Sort, Exchange, Rotation) and we want to give some estimates on their total variation. We remark here that the $\BV$ regularity of the vector fields make possible all these constructions, allowing for rigid \emph{cut and paste motions}, in the spirit of \cite{Bianchini_Zizza_residuality} and \cite{univ:mixer}. \\

\noindent Following the notation introduced in \cite{Bianchini_Zizza_residuality} we define a flow that rotates subrectangles of $\mathbb{T}^2$ and the vector field associated with it.  More precisely we define the \emph{rotation flow} $r_{t}:\mathbb{T}^2\rightarrow \mathbb{T}^2$ for $t\in[0,1]$ in the following way: call
\begin{equation*}
	V(x)=\max \left\{ \left| x_1-\frac{1}{2}\right|, \left| x_2-\frac{1}{2}\right|\right\}^2, \quad (x_1,x_2)\in \mathbb{T}^2.
\end{equation*}
Then the \emph{rotation field} is $r:\mathbb{T}^2\rightarrow \R^2$
\begin{equation}
	\label{rotation field}
	{r}(x)=\nabla^\perp V(x),
\end{equation}
where $\nabla^\perp =(-\partial_{x_2},\partial_{x_1})$ is the orthogonal gradient. Finally the rotation flow $r_{t}$ is the flow of the vector field $r$, i.e. the unique solution to the following ODE system:
\begin{equation}
\label{rot:flow:square}
	\begin{cases} 
	\dot {r}(t,x)={r}(r(t,x)), \\
	r(0,x)=x.
	\end{cases}
\end{equation}
This flow rotates the unit square $[0,1]^2$ counterclockwise of an angle $\frac{\pi}{2}$ in a unit interval of time. We recall here Lemma 2.6 in \cite{Bianchini_Zizza_residuality} that gives the cost, in terms of the total variation, of the rotation of a rectangle:
\begin{lemma}\label{TV:rotations}
Let $Q\subset[0,1]^2$ be a rectangle of sides $a,b>0$. Consider the rotating flow on the torus 
\begin{equation*}
%\label{rotating flow}
    (R_Q)(t,x)=\chi_Q^{-1}\circ r(t)\circ \chi_Q(x), \quad\text{if }x\in Q.
\end{equation*}
where $\chi_Q:Q\rightarrow [0,1]^2$ is the affine map sending $Q$ into $[0,1]^2$ and $r(t)$ is the rotation flow defined in \eqref{rot:flow:square}. Let $b(t)^{R_Q}$ the divergence-free vector field associated with $(R_{Q})(t)$ (extended to $0$ outside the rectangle $Q$). Then \begin{equation*}
%\label{TV:rot}
\TV(b^R(t))(\mathbb{T}^2) = 4 a^2 + 4 b^2, \quad \forall t\in [0,1].
\end{equation*}
\end{lemma}
 \noindent
 Let us fix now some $k\in\N$ giving the size of the partition, and let us consider the grid given by $\N\times\N\frac{1}{k}$ made of squares of side $\frac{1}{k}$. We consider the subsquares $Q_{i,j}=[\frac{j-1}{k},\frac{j}{k}]\times[\frac{k-i}{k},\frac{k-i+1}{k}]$ and we identify each subsquare with  an entry $\gamma_{ij}$ of some configuration $\gamma\in\mathcal{C}(k)$. A horizontal stripe $H_i=[0,1]\times[\frac{k-i}{k},\frac{k-i+1}{k}]$ is a row, while a vertical stripe $V_j=[\frac{j-1}{k},\frac{j}{k}]\times [0,1]$ is a column. We observe that if $Q_{i,j}$ and $Q_{i',j'}$ are adjacent subsquares (that is, they share a common side), then $\gamma_{i,j}$ and $\gamma_{i',j'}$ are adjacent entries (Definition \ref{adjacent}). From now on we will identify every $\gamma\in\mathcal{C}(k)$ with an enumeration of the subsquares of the torus, and any movement $T:\mathcal{C}(k)\rightarrow\mathcal{C}(k)$ with an automorphism $T:\T^2\rightarrow\T^2$that sends rigidly every subsquare of the grid $\N\times\N\frac{1}{k}$ into another one of the same grid. We start introducing the vector fields whose RLF, when evaluated at time $t=1$, are movements. \\
 
 \noindent We first recall the transposition vector field, first introduced in \cite{Bianchini_Zizza_residuality}. 
 Let $\kappa_i,\kappa_j$ be two adjacent squares of size $\frac{1}{k}$ and let $Q=\kappa_i\cup \kappa_j$, then the \emph{transposition flow} between $\kappa_i,\kappa_j$ is $T(t)(\kappa_i,\kappa_j):[0,1]\times \T^2\rightarrow \T^2$ defined as
\begin{equation}
	\label{transposition flow}
		T(t)(\kappa_i,\kappa_j)=
	\begin{cases}
	\chi^{-1}\circ r(4t)\circ \chi & x\in \mathring{Q}, \ t\in\left[0,\frac{1}{2}\right], \\
		\chi_i^{-1}\circ r(4t)\circ \chi_i & x\in\mathring{\kappa}_i, \ t\in\left[\frac{1}{2},1\right], \\
		\chi_j^{-1}\circ r(4t)\circ \chi_j & x\in\mathring{\kappa}_j, \ t\in\left[\frac{1}{2},1\right], \\
		x & \text{otherwise},
	\end{cases}
\end{equation}
where the map $\chi:Q\rightarrow \T^2$ is the affine map sending the rectangle $Q$ into the torus $\T^2$, $\chi_{i},\chi_{j}$ are the affine maps sending $\kappa_i,\kappa_j$ into the torus $\T^2$ and $r$ is the rotation flow (\ref{rot:flow:square}). This invertible measure-preserving flow has the property to exchange the two subsquares in the unit time interval (Figure \ref{fig:transp}). Moreover, by the computations done in  Lemma \ref{TV:rotations}, we can estimate the total variation of the vector field $b^T(t)(\kappa_i,\kappa_j)$ associated with $ T(t)(\kappa_i,\kappa_j)$ (recall \ref{vect:field:ass}) as
\begin{equation}
\label{TV:trasp}
	\TV(b^T(t)(\kappa_i,\kappa_j))(\T^2)\leq 4\frac{20}{k^2}.
\end{equation}
We also observe that also $L^\infty$ estimates on the vector field are easily available: indeed
\begin{equation}\label{inf:transpo}
    \|b^T(\kappa_i,\kappa_j)\|_{L^\infty_{t.x}}\leq 4\frac{2}{k}.
\end{equation}
\begin{figure}
    \centering
    \includegraphics[scale=0.6]{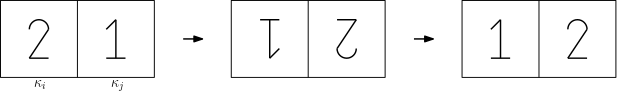}
    \caption{A transposition between two squares is an exchange.}
    \label{fig:transp}
\end{figure}
 \begin{definition}[Simple exchange vector field] Let us fix $i,i'\leq k$ with $|i-i'|=1$ and $j\leq k$ (or alternatively $i\in\N$ and $j,j'\in\N$ with $|j-j'|=1$). The simple exchange \emph{vector field} is  $b(t)(i,j;i',j)\in L^\infty_t\BV_x$ such that, if $X(i,j;i',j)$ is its RLF evaluated at time $t=1$, then
 \begin{equation}
     X(i,j;i',j)(\T^2)=E_s(i,j;i',j)(\gamma),\quad\forall\gamma\in\mathcal{C}(k).
 \end{equation}
 The construction of this vector field is easy: fix for example $i'=i+1$. Then take the two adjacent subsquares $Q_{ij}$ and $Q_{i+1,j}$ and perform a transposition between them $T(t)(Q_{i,j},Q_{i+1,j})$. Then define
 \begin{equation*}
     b(t)(i,j;i+1,j)=b^T(t)(Q_{ij},Q_{i+1,j}).
 \end{equation*}
 Clearly one has 
\begin{equation}\label{est:simple:exch}
     \|b(i,j;i+1,j)\|_{L^\infty_{t,x}}\leq \frac{8}{k},\qquad\TV(b(t)(i,j;i+1,j))(\T^2)\leq 4\frac{20}{k^2}.
 \end{equation}
 \end{definition}

\noindent Similarly we have \begin{definition}[Sort vector field]
Let $i\leq k$ be some fixed row index. Let $j,j'\leq k$ with $j<j'$. Then the \emph{sort vector field} (on columns) is $b^{S_c}(i;j,j')(t)\in L^\infty_t\BV_x$ such that, if $X^{S_c}(i;j,j')$ is its RLF evaluated at time $t=1$ then 
\begin{equation}
   X^{S_c}(i;j,j')(\T^2)=S_c(i;j,j')(\gamma),\quad\forall\gamma\in\mathcal{C}(k).
\end{equation}
\end{definition}
\begin{remark}
Similarly one has the sort vector field on rows $b^{S_r}(i,i';j)(t)$ for some $i<i'\leq k$, $j\leq k$.
\end{remark}
\noindent To construct a sort vector field (on columns, but the construction for rows is identical), we fix  $\gamma\in\mathcal{C}(k)$, a row index $i$ and two columns indices $j<j'$. The idea is to perform $j'-j$ exchanges between squares (recalling that an exchange is a transposition). For clarity see Figure \ref{fig:sort:columns} and compare with Example \ref{ex:sort}.
\begin{figure}
    \centering
    \includegraphics[scale=0.45]{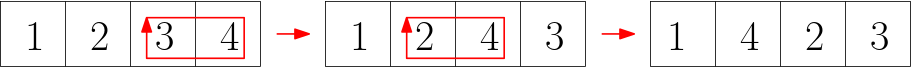}
    \caption{\small{An example of a sort vector fields on a horizontal stripe of the torus. Here $j=2,j'=4$. The transpositions to be performed are $j'-j=2$.}}
    \label{fig:sort:columns}
\end{figure}
\noindent We define the flow in the following way:
\begin{equation}
    {X}^{S_c}(i;j,j')(t)=\begin{cases}
    T_{(j'-j)t}\left(Q_{i(j'-1)},Q_{ij'}\right) \quad\text{for } t\in\left[0,\frac{1}{j'-j}\right], \\
    
    T_{(j'-j)t-1}\left(Q_{i(j'-2)},Q_{i(j'-1)}\right) \quad\text{for } t\in\left[\frac{1}{j'-j},\frac{2}{j'-j}\right], \\
    \dots \\
    T_{(j'-j)t-(j'-j-2)}\left(Q_{i(j+1)},Q_{i(j+2)}\right) \quad\text{for } t\in\left[\frac{j'-j-2}{j'-j},\frac{j'-j-1}{j'-j}\right], \\
    T_{(j'-j)t-(j'-j-1)}\left(Q_{ij},Q_{i(j+1)}\right) \quad\text{for } t\in\left[\frac{j'-j-1}{j'-j},1\right],  \\
    x\quad\text{otherwise}.
    \end{cases}
\end{equation}
By definition one has 
\begin{equation}\label{eq:est:flow:2}
    b^{S_c}(i;j,j')(t)=\dot{X}^{S_c}(i;j,j')(t)\circ {X}^{S_c}(i;j,j')^{-1}(t),
\end{equation}
so that 
\begin{equation*}
    \sup_{t,x} |b^{S_c}(i;j,j')(t,x)|\leq \sup_{t,x} |\dot{X}^{S_c}(i;j,j')(t,x)|\leq (j-j')\frac{8}{k}\leq 8,
\end{equation*}
being the rotation vector field of the order of the side of subsquares. This implies that
\begin{equation}\label{est:sort:1}
    \|b^{S_c}(i;j,j')\|_{L^\infty_t L^\infty_x}\leq 8.
\end{equation}

\noindent Similarly one has, for every $t\in[0,1]$
\begin{equation}\label{est:sort:2}
    \TV(b^{S_c}(t)(i;j,j'))(\T^2)\leq (j'-j)\cdot 4\cdot 4\left(\frac{4}{k^2}+\frac{1}{k^2}\right)\leq \frac{80}{k}.
\end{equation}
\begin{remark} Recalling Remark \ref{rk:shift:vs:sort} we observe that for the shift operation (that is $j'=k$ and $j=1$) we can consider another vector field, exploiting the structure of the torus. We will still denote it as the sort vector field, but
\begin{equation}\label{eq:shift:vf}
    b(t,x)(i;1,k)=\begin{cases}
    \frac{1}{k}, \quad \text{if } x\in H_{i}, \\
    0\quad\text{otherwise}.
    \end{cases}
\end{equation}
So we have instead the following estimates:
\begin{equation}\label{est:shift}
    \|b(i,1,k)\|_{L^\infty_{t,x}}\leq \frac{1}{k}, \qquad \|\TV(b(i,1,k)(\T^2)\|_{\infty}\leq 2\frac{1}{k}. 
\end{equation}
This tells us that both the shift vector field and the sort vector field have total variation of the order of the side of the squares of the grid.
\end{remark}
\begin{definition}[Rotation vector field]
 Let $2\leq i\leq k$ and $j\leq k-1$. Then the counterclockwise rotation vector field is $ b^-(i,j)\in L^\infty_t\BV_x$ such that, if $ {X}^-(i,j)$ is its RLF evaluated at time $t=1$, then
\begin{equation}
    {X}^-(i,j)(\T^2)=R^-_{ij}(\gamma),\quad\forall\gamma\in\mathcal{C}(k). 
\end{equation}

\noindent The clockwise rotation vector field is $ b^+(i,j)\in L^\infty_t\BV_x$ such that, if $ {X}^+(i,j)$ is its RLF evaluated at time $t=1$, then
\begin{equation}
    {X}^+(i,j)(\T^2)=R^+_{ij}(\gamma),\quad\forall\gamma\in\mathcal{C}(k).
\end{equation}
\end{definition}
\noindent We write here just the counterclockwise case since the other one is identical. Here we have that, if we fix $i,j$ as in the definition, we call $Q=Q_{(i-1)j}\cup Q_{ij}\cup Q_{i(j+1)}\cup Q_{(i-1)(j+1)}$, then 
\begin{equation}
    X^-(i,j)(t)=\begin{cases}
    \chi_Q^{-1}\circ r_{2t}\circ\chi_Q,\quad x\in\mathring{Q},t\in\left[0,\frac{1}{2}\right], \\
    \chi_{lm}^{-1}\circ r^{-1}_{2t}\circ\chi_{lm},\quad x\in\mathring{Q_{lm}},t\in\left[\frac{1}{2},1\right], l=i, m=j, j+1, \text{ or } l=i-1, m=j,j+1, \\
    x \quad \text{ otherwise, }
    \end{cases}
\end{equation}
where $\chi$ is the affine map sending $Q$ into $\T^2$ and $\chi_{lm}$ is the affine map sending $Q_{lm}$ into $\T^2$.
Again, one has
\begin{equation}\label{eq:est:flow:3}
|\dot X^{-}(i,j)(t,x)|\leq \frac{2}{k},
    \end{equation}
    from which one gets
\begin{equation}\label{est:rot}
    \|b^-(i,j)\|_{L^\infty_{t,x}}\leq\frac{2}{k},\qquad  \TV(b^-(t)(i,j))(\T^2)\leq 2\cdot 4\left(\frac{4}{k^2}+\frac{4}{k^2}\right),\quad\forall t.
\end{equation}

\noindent In the Figure \ref{fig:rotations} an example of the action of a clockwise rotation vector field. 

\begin{figure}[H]
    \centering
    \includegraphics[scale=0.7]{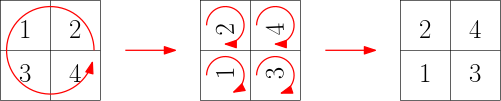}
    \caption{Action of the counterclockwise rotation vector field, where $k=2$, $i=2$, $j=1$.}
    \label{fig:rotations}
\end{figure}
\begin{remark}
If $X:[0,1]\times\T^2\rightarrow T^2$ is one of the flows previously considered (simple exchange, sort/shift, rotation), then, if $k$ is the side of the squares in which the torus is tiled, one has
\begin{equation}\label{eq:est:flow:final}
    \| \dot{X}\|_{L^\infty_tL^\infty_x}\leq\max\left(\frac{8}{k},\frac{(j-j')8}{k},\frac{2}{k}\right)
\end{equation}
We remark also that, if $b$ is a vector field associated with a movement, then, $\forall t\in[0,1]$, it holds
\begin{equation}\label{eq:est:flow:final:TV}
    \| \TV(b)(\T^2)\|_{\infty}\leq \max \left(\frac{80}{k^2},(j-j')\frac{80}{k^2},\frac{64}{k^2}\right),
\end{equation}
which follows by the second of \eqref{est:simple:exch},\eqref{est:sort:2} and the second of \eqref{est:rot}. 
\end{remark}

\section{Two dimensional construction}\label{S:twod}
Our aim is the construction of a weakly mixing vector field which is not strongly mixing, that is a divergence-free vector field $b\in L^\infty([0,1],\BV(\mathbb{T}^2))$ whose RLF $X\llcorner_{t=1}$ when evaluated at time $t=1$ is a weakly mixing automorphism of $\mathbb{T}^2$ but not a strongly mixing automorphism. The idea is to adapt Chacon's one-dimensional construction of Subsection \ref{Ss:chacon} and decompose the map into simple movements (Section \ref{subS_rows_columns}), described as the flow of some divergence-free BV vector field (Subsection \ref{Ss:vf}). 
  Also in this case the fundamental idea is to define $U$  as the limit of a family of automorphisms $\lbrace U_k\rbrace_k$. \\
 
 \noindent Let us consider the two dimensional torus $\T^2$ as the unit square $Q_1=Q_{1,1}=[0,1]^2$ with the canonical identification of boundaries. Using the notation via configurations (see Section \ref{subS_rows_columns}) we say that $Q_1$ can be identified with the configuration $\gamma_1=(1)\in\mathcal{C}(1)$. We define $h_1=0$ and we divide the square $Q_1=Q_{1,1}$ into four identical subsquares each one of side $\frac{1}{2}$, more precisely: $Q_{1,1}=Q_{1,1}(1)\cup Q_{1,1}(2)\cup Q_{1,1}(3)\cup R_1$ where $Q_{1,1}=[0,\frac{1}{2}]\times[\frac{1}{2},1]$, $Q_{1,1}(2)=[\frac{1}{2},1]\times[\frac{1}{2},1]$ and $Q_{1,1}(3)=[0,\frac{1}{2}]\times [0,\frac{1}{2}]$. Clearly $R_1=[\frac{1}{2},1]\times[0,\frac{1}{2}]$. We define $U_1$ on these subsquares in such a way that 
 \begin{equation}
     U_1(Q_{1,1}(1))=Q_{1,1}(2), \quad U_{1}(Q_{1,1}(2))=Q_{1,1}(3)
 \end{equation}
 and $U_1$ is measure-preserving and invertible on $Q_{1,1}$. More precisely, we define $U_1$ as 
 \begin{equation}
     U_1(x)=\begin{cases}
     x+(\frac{1}{2},0) \quad\text{if}\quad x\in \mathring{Q}_{1,1}(1), \\
     x+(-\frac{1}{2},-\frac{1}{2}) \quad\text{if}\quad x\in \mathring{Q}_{1,1}(2),\\
     x+(0,\frac{1}{2}) \quad\text{if}\quad x\in \mathring{Q}_{1,1}(3), \\
     x \quad\text{otherwise}.
     \end{cases}
 \end{equation}
 
 \noindent We put $h_2=4h_1+3$ and we rename the subsquares $Q_{1,1}(i)$ as $Q_{2,1}=Q_{1,1}(1), Q_{2,2}=Q_{1,1}(2), Q_{2,3}=Q_{1,1}(3)$. We look at this via configurations: let us take $\gamma_2\in\mathcal{C}(2)$, for example
 \begin{equation}\label{eq:gamma:2}
     \gamma_2=\left(\begin{matrix} 1 & 2 \\ 3 & 4
     \end{matrix}\right).
 \end{equation}
 Then we can easily represent $U_1(\gamma_2)$ as
 \begin{equation}\label{U:1:gamma}
     U_1(\gamma_2)=\left(\begin{matrix} 3 & 1 \\ 2 & 4\end{matrix}\right).
 \end{equation}
Here we see the advantage of the representation via configurations, which is more immediate.  We continue our construction dividing each subsquare $Q_{2,i}$ with $i\in\lbrace{1,2,3}\rbrace$ into $4$ subsquares of side equal to $\frac{1}{4}$: $Q_{2,i}=Q_{2,i}(1)\cup Q_{2,1}(2)\cup Q_{2,1}(3)\cup Q_{2,i}(4)$ and we divide also $R_1=R_{1}(1)\cup R_{1}(2)\cup R_{1}(3)\cup R_2$ into $4$ subsquares (see Figure \ref{twodsub}). 
 \begin{figure}
     \centering
     \includegraphics[scale=0.5]{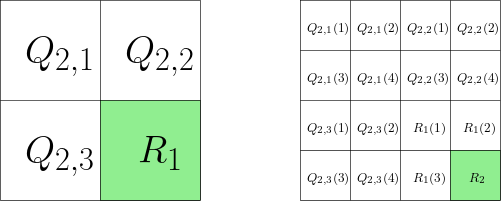}
     \caption{Subdivisions of the two dimensional torus.}
     \label{twodsub}
 \end{figure}
 We define the map $U_2$ in the following way
\begin{equation}
    U_2(x)=\begin{cases}
    U_1(x) \quad\text{if}\quad x\in \mathring{Q}_{2,1}\cup \mathring{Q}_{2,2}, \\
    x +(\frac{1}{4},\frac{1}{2}) \quad\text{if}\quad x\in \mathring{Q}_{2,3}(1), \\
    x +(\frac{1}{4},0) \quad\text{if}\quad x\in \mathring{Q}_{2,3}(2) ,\\
    x +(\frac{3}{4},\frac{1}{4}) \quad\text{if}\quad x\in \mathring{Q}_{2,3}(3), \\
    x +(\frac{1}{4},0) \quad\text{if}\quad x\in \mathring{Q}_{2,3}(4), \\
     x +(-\frac{1}{2},\frac{1}{4})  \quad\text{if}\quad x\in \mathring{R}_1(1), \\
    x +(-\frac{1}{2},\frac{1}{4}) \quad\text{if}\quad x\in \mathring{R}_1(2), \\
    x +(-\frac{1}{2},\frac{3}{4}) \quad\text{if}\quad x\in \mathring{R}_1(3), \\
     x \quad\text{otherwise},
    \end{cases}
\end{equation}
or, via configurations, fixing some $\gamma_4\in\mathcal{C}(4)$ we get
\begin{equation}\label{eq:U_2}
U_2\left(\begin{matrix}
1 & 2 & 5 & 6 \\ 3 & 4 & 7 & 8 \\ 9 & 10 & 13 & 14 \\ 11 & 12 & 15 & 16
\end{matrix}\right)=\left(\begin{matrix}
15 & 9 & 1 & 2 \\ 13 & 14 & 3 & 4 \\ 5 & 6 & 10 & 11 \\ 7 & 8 & 12 & 16
\end{matrix}\right).
\end{equation}

\noindent We put $h_3=4h_2+3=15$ and we rename the squares as \begin{align*}
    Q_{3,1}\doteq Q_{2,1}(1)&\to Q_{3,2}\doteq Q_{2,2}(1)\to Q_{3,3}\doteq Q_{2,3}(1)\to \\ & \to Q_{3,4}=Q_{2,1}(2)\to\dots\to  Q_{3,14}\doteq Q_{2,3}(3)\to Q_{3,15}\doteq R_1(3),
\end{align*}
with the property that if $Q_{3,i}\to Q_{3,i+1}$ then $U_2(Q_{3,i})=Q_{3,i+1}$. \\
 
\noindent The inductive step is the following: at the $n$-th step we have $Q_{n,1},\dots Q_{n,h_{n}}$ subsquares each one of area $\frac{1}{4^{n-1}}$ and side $l_n=\frac{1}{2^{n-1}}$. We divide each subsquare into 4 identical subsquares: $Q_{n,i}=Q_{n,i}(1)\cup Q_{n,i}(2)\cup Q_{n,i}(3)\cup Q_{n,i}(4)$ and we divide also $R_{n-1}$ into 4 identical subsquares $R_{n-1}=R_{n-1}(1)\cup R_{n-1}(2)\cup R_{n-1}(4)\cup R_n$. We put $h_{n+1}=4h_{n}+3=4^n-1$ and we define the map $U_{n}$ in the following way:
\begin{equation}\label{map:U:n:voglio:mori}
    U_{n}(x)=\begin{cases}
    U_{n-1}(x) \text{ if } x\in \mathring{Q}_{n,1}\cup\dots\cup \mathring{Q}_{n,h_n-1}, \\
    x+ (\frac{5}{4^{n-1}}-1,1-\frac{1}{2^{n-1}})\text{ if } x\in \mathring{Q}_{n,h_n}(1), \\
    x+ (\frac{1}{4^{n-1}},0)\text{ if } x\in \mathring{Q}_{n,h_n}(2), \\
    x+ (\frac{3}{4^{n-1}},\frac{1}{4^{n-1}})\text{ if } x\in \mathring{Q}_{n,h_n}(3), \\
    x+ (\frac{1}{4^{n-1}},0)\text{ if } x\in \mathring{Q}_{n,h_n}(4), \\
    x+ (-1+\frac{1}{2^{n-1}},\frac{2^n-3}{4^{n-1}})\text{ if } x\in \mathring{R}_{n-1}(1), \\
    x+ (\frac{1}{2^{n-1}}-1,\frac{2^n-3}{4^{n-1}})\text{ if } x\in \mathring{R}_{n-1}(2), \\
    x+ (\frac{1}{2^{n-1}}-1,\frac{2^n-1}{4^{n-1}})\text{ if } x\in \mathring{R}_{n-1}(3), \\
    x \text{ otherwise }.
    \end{cases}
\end{equation}

\noindent We underline that, by definition, we have that
    \begin{equation}\label{eq:inclusion:maps}
        U_n(x)=U_{n-1}(x) \text{ if } x\in Q_{n,1}\cup\dots\cup Q_{n,h_n-1}.
    \end{equation}
Finally we rename the squares as follows:
\begin{align*}
    Q_{n+1,1}\doteq Q_{n,1}(1)&\to Q_{n+1,2}\doteq Q_{n,2}(1)\to\dots Q_{n+1,h_n}\doteq Q_{n,3}(1)\to \\ & \to\dots\to  Q_{n+1,4h_n+2}\doteq Q_{n,h_n}(3)\to Q_{n+1,h_{n+1}}\doteq R_n(3),
\end{align*}
with the property that if $Q_{n+1,i}\to Q_{n+1,i+1}$ then $U_n(Q_{n+1,i})=Q_{n+1,i+1}$. \\

\noindent finally we define $U=\lim_{n\to\infty} U_n$ (well defined by condition \ref{eq:inclusion:maps}).
Then the following propositions hold:
\begin{proposition}
The map $U$ is measure-preserving and ergodic.
\end{proposition}
\noindent The proof of the proposition relies on the notion of \emph{sufficient semi-rings}, that can be taken from \cite{Chen2015THENO}. 
 \begin{proposition}\label{prop:wm:not:sm}
 The map $U$ is a weakly mixing automorphism of $\T^2$ that is not strongly mixing.
 \end{proposition}
 \begin{proof}
 The idea of the proof is the same of Chacon (see   \cite{Weaklymixingnotstrongly}). Indeed: take for example a subsquare $Q_n=Q_{n,i}$ with $n$ sufficiently large. We will prove that 
 \begin{equation}
     |U^{h_k}(Q_n)\cap Q_n|= \frac{1}{4} |Q_n|> |Q_n|^2, \quad\forall k\geq n.
 \end{equation}
 We first recall that $U=\lim_{n\to\infty} U_n$ and that $U_{n+1}=U_{n}$ on $Q_{n+1,1}\cup Q_{n+1,2}\cup \dots\cup Q_{n+1,h_{n+1}-1}$. %Observe that
 %\begin{equation}\label{eq:inclusions}
  %   Q_{n,1}\cup\dots \cup Q_{n,h_n}\subset Q_{n+1,1}\cup Q_{n+1,2}\cup \dots\cup Q_{n+1,h_{n+1}-1}.
 %\end{equation}
 We remember that $Q_n=Q_{n,i}(1)\cup Q_{n,i}(2)\cup Q_{n,i}(3)\cup Q_{n,i}(4)$, so
 \begin{align*}&U_n^{h_n}(Q_{n,i}(1))=Q_{n,i}(2), \\ & U_n^{h_n}(Q_{n,i}(2))=Q_{n,i-1}(3), \\ & U_n^{h_n}(Q_{n,i}(3))=Q_{n,i-1}(4), \\ & U_n^{h_n}(Q_{n,i}(4))=Q_{n,i-1}(1).
 \end{align*}
 This implies that
 \begin{equation*}
     | U^{h_n}_n(Q_n)\cap Q_n|=\frac{1}{4}|Q_n|,
 \end{equation*}
 for any $Q_n$ sublevel of the $n$-th column. Since \eqref{eq:inclusion:maps} holds and also $U_n^{h_n}({Q_{n,i}(1)})=Q_{n,i}(2)$, one has
 \begin{equation*}
     |U^{h_n}(Q)\cap Q|\geq\frac{1}{4}|Q|.
 \end{equation*}
 In particular, if one considers $k\geq n$ and takes a sublevel $Q_k$ of the $k$-th column, one gets similarly
 \begin{equation*}
     |U^{h_k}(Q_k)\cap Q_k|\geq\frac{1}{4}|Q_k|.
 \end{equation*}
 We observe now that any level $Q_n$ of the $n$-th column has $4^{k-n}$ copies into the $k$-th column. For example, if $k=n+1$ one gets 
 $Q_n=Q_{n+1}^{1}\cup Q_{n+1}^{2}\cup Q_{n+1}^{3}\cup Q_{n+1}^{4}$. But then
 \begin{equation*}
     |U^{h_{n+1}}(Q_n)\cap Q_n|=\sum_{j=1}^4|U^{h_{n+1}}(Q_{n+1}^j)\cap Q_n|\geq \sum_{j=1}^4|U^{h_{n+1}}(Q_{n+1}^j)\cap Q_{n+1}^j|\geq\sum_{j=1}^4\frac{1}{4}|Q_{n+1}^j|=\frac{1}{4}|Q|.
 \end{equation*}
 More in general
  \begin{equation*}
     |U^{h_{k}}(Q_n)\cap Q_n|=\sum_{Q'\text{ copies }}|U^{h_{k}}(Q')\cap Q_n|\geq \sum_{Q'\text{ copies }}|U^{h_{k}}(Q')\cap  Q'|\geq\sum_{Q'\text{ copies }}\frac{1}{4}|Q'|=\frac{1}{4}|Q|.
 \end{equation*}
  This gives a diverging sequence $\lbrace k\rbrace$ and a set $Q_n$ for which the strongly mixing condition does not hold (if one simply requires that $|Q_n|<\frac{1}{4}$),
 therefore $U$ cannot be strongly mixing (see Definition \ref{mix}).
 Whereas $U$ is weakly mixing. Indeed, by the previous proposition, $U$ is measure-preserving and ergodic. By Theorem \ref{mixing theorem} we have to prove that if $f\in\L^2(\T^2)$ is an eigenfunction of $\mathcal{U}_U$ of eigenvalue $\lambda$, with $|\lambda|=1$, then $f=c$ for some constant $c$. By Lemma $2$ of \cite{Ergodic:theory} (chapter 1, page 14), we need only to prove that $\lambda=1$. Since $f$ is non-zero on a set of positive measure, for every $\epsilon>0$ there exists a constant $k$ such that the set
$$ A=\lbrace x: |f(x)-k|<\epsilon\rbrace $$ has positive measure. By the property of sufficient semi-rings (see again \cite{Chen2015THENO}) there exists a subsquare $Q=Q_{n,i}$ such that 
\begin{equation}\label{density}
    \frac{|Q\cap A|}{|Q|}> \frac{7}{8}.
\end{equation}
As in the previous computations, since $Q=Q_{n,i}(1)\cup Q_{n,i}(2)\cup Q_{n,i}(3)\cup Q_{n,i}(4)$, using that $U$ is measure-preserving, $U^{h_n}(Q_{n,i}(1))=Q_{n,i}(2)$, $U^{2h_n+1}(Q_{n,i}(1))=Q_{n,i}(3)$ and that the inequality \eqref{density} holds, we obtain that there exists a point $x\in A\cap Q$ such that $U^{h_n}(x)\in A\cap Q$, and that there exists $y\in A\cap Q$ such that $U^{2h_n+1}(y)\in A\cap Q$. Indeed, by \eqref{density} we have that $|Q_{n,i}(2)\cap A|>\frac{1}{8}|Q|$ and for the same reason that
$$ |Q_{n,i}(1)\cap A|>\frac{1}{8}|Q|, $$ therefore, by applying the map $U^{h_n}$ one gets
$$ |Q_{n,i}(2)\cap U^{h_n}(A)|>\frac{1}{8}|Q| ,$$ which implies that $$ |(Q\cap A)\cap U^{h_n}(Q\cap A)|>0.$$ An analogous argument proves that $$ |(Q\cap A)\cap U^{2h_n+1}(Q\cap A)|>0.$$
Therefore 
\begin{equation}
    |f(x)-k|<\epsilon, \quad |\lambda^{h_n}f(x)-k|<\epsilon, \quad |f(y)-k|<\epsilon, \quad
    |\lambda^{2h_n+1}f(y)-k|<\epsilon.
\end{equation}
So by the previous equations we have that
\begin{equation*}
    \lambda^{h_n}= \frac{k+\delta_2}{k+\delta_1}, \quad \lambda^{2h_n+1}=\frac{k+\delta_4}{k+\delta_3}
\end{equation*}
with $|\delta_i|<\epsilon$ for every $i$. That is
\begin{equation}
    \lambda=\frac{(k+\delta_4)(k+\delta_1)^2}{(k+\delta_3)(k+\delta_2)^2},
\end{equation}
and since this hold for every $\epsilon>0$, we conclude that $\lambda=1$. 
 \end{proof}
 
 \noindent The maps $U_n$ can be decomposed into simple movements, using the notations of configurations.  \\
 
 \noindent
 We have the following
 \begin{proposition}\label{prop:che:palle} For every $n\geq 2$ the map $U_n$ can be expressed as 
 \begin{equation}\label{eq:recursive:map}
     U_{n}=U_{n-1}\circ V_{n},
 \end{equation}
 where $V_{n}:\mathcal{C}(2^{n})\rightarrow\mathcal{C}(2^{n})$ and
 \begin{align}\label{V:n}
\notag V_{n}= & E_s(2^{n},2^{n}-3;2^{n},2^{n}-2)\circ R^+(2^{n},2^{n}-3)\circ  R^+(2^{n},2^{n}-3)\circ \\& \circ S_c^2(2^{n}-1;2^{n}-3,2^{n})\circ R^-(2^{n},2^{n}-3)\circ S_c(2^{n};2^{n}-3,2^{n}-1).
 \end{align}
\end{proposition}
\begin{proof}
 For understanding the situation we first fix $n=1$, and the starting configuration
 \begin{equation*}
     \gamma=\left(\begin{matrix}1 & 2 & 5 & 6 \\ 3 &4 & 7 & 8 \\ 9 & 10 & 13 & 14 \\ 11 &12 &15 &16\end{matrix}\right),
 \end{equation*}
 We first apply $V_1$ finding
 \begin{align*}
     & \left(\begin{matrix}1 & 2 & 5 & 6 \\ 3 &4 & 7 & 8 \\ 9 & 10 & 13 & 14 \\ 11 &12 &15 &16\end{matrix}\right)\stackrel{S_c(4;1,3)}{\longrightarrow} \left(\begin{matrix}1 & 2 & 5 & 6 \\ 3 &4 & 7 & 8 \\ 9 & 10 & 13 & 14 \\ 15 &11 &12 &16\end{matrix}\right)\stackrel{R^-(4,1)}{\longrightarrow}\left(\begin{matrix}1 & 2 & 5 & 6 \\ 3 &4 & 7 & 8 \\ 10 & 11 & 13 & 14 \\ 9 &15 &12 &16\end{matrix}\right)\stackrel{S_c^2(3;1,4)}{\longrightarrow} \\& \notag \\ &
     \left(\begin{matrix}1 & 2 & 5 & 6 \\ 3 &4 & 7 & 8 \\ 13 & 14 & 10 & 11 \\ 9 &15 &12 &16\end{matrix}\right)\stackrel{(R^+(4,1))^2}{\longrightarrow}\left(\begin{matrix}1 & 2 & 5 & 6 \\ 3 &4 & 7 & 8 \\ 15 & 9 & 10 & 11 \\ 14 &13 &12 &16\end{matrix}\right)\stackrel{E_s(4,1;1,2)}{\longrightarrow}\left(\begin{matrix}1 & 2 & 5 & 6 \\ 3 &4 & 7 & 8 \\ 15 & 9 & 10 & 11 \\ 13 &14 &12 &16\end{matrix}\right).
 \end{align*}
 Then by applying $U_1$ to the last configuration
 \begin{equation*}
     \left(\begin{matrix}1 & 2 & 5 & 6 \\ 3 &4 & 7 & 8 \\ 15 & 9 & 10 & 11 \\ 13 &14 &12 &16\end{matrix}\right)\stackrel{U_1}{\longrightarrow}\left(\begin{matrix}15 & 9 & 1 & 2 \\  13&14 & 3 & 4 \\ 5 & 6 & 10 & 11 \\ 7 &8 &12 &16\end{matrix}\right),
 \end{equation*}
 to compare with \eqref{eq:U_2}. For a generic $n\in\N$, one has to observe that, recalling \eqref{eq:inclusion:maps}, one has
  \begin{equation}\label{eq:squares}
        U_{n}(x)=U_{n-1}(x) \text{ if } x\in Q_{n,1}\cup\dots\cup Q_{n,h_{n}-1}.
    \end{equation}
 with $h_{n}=4^{n-1}-1$.
 We list a series of useful observations:
 \begin{itemize}
     \item consider a configuration $\gamma\in\mathcal{C}(2^{n})$, then the map $V_n$ acts only on the subsquares 
     \begin{align}\label{series:squares}
         Q_{n+1,h_n}, Q_{n+1,2h_n}, Q_{n+1,2h_n+1},Q_{n+1,3h_n+1}, Q_{n+1,3h_n+2},Q_{n+1,4h_n+2},Q_{n+1,4h_{n}+3},R_{n} ;
     \end{align}
     \item $Q_{n,1}\cup\dots\cup Q_{n,h_n-1}$ in equation \ref{eq:squares} correspond to the following subsquares of the refined grid
 \begin{align*}
     & Q_{n+1,1},Q_{n+1,2},\dots,Q_{n+1,h_n-1}, Q_{n+1,h_n+1},Q_{n+1,h_n+2},\dots, Q_{n+1,2h_n-1}, \\ & \notag \\ & Q_{n+1,2h_n+2},Q_{n+1,2h_n+3},\dots, Q_{n+1,3h_n-1},Q_{n+1,3h_n},\dots Q_{n+1,3h_n+3},\\& \notag \\& Q_{n+1,3h_n+4},\dots Q_{n+1,4h_n+1};
    \end{align*}
 \item therefore we have to check $U_{n-1}\circ V_n$ only on the subsquares of \eqref{series:squares};
 \item we recall that, by the enumeration chosen,
 $U_n(Q_{n+1,i})=Q_{n+1,i+1}$ and $U_n(Q_{n+1,h_{n+1}})=Q_{n+1,1}$;
 \item we recall also that
 \begin{equation}\label{obs}U_{n-1}(Q_{n,h_n})=Q_{n,1}=Q_{n+1,1}\cup Q_{n+1,h_n+1}\cup Q_{n+1,2h_n+2} \cup Q_{n+1,3h_n+3},\end{equation}
 %\begin{figure}[H]
 %\centering
 %\includegraphics[scale=0.5]{square1.png}
 %\end{figure}
 and 
 \begin{equation}
     \label{obs:2}
     U_{n-1}\llcorner R_{n-1}=Id.
 \end{equation}
 That is, since $Q_n=Q_{n+1,h_n}\cup Q_{n+1,2h_n}\cup Q_{n+1,3h_n+1}\cup Q_{n+1,4h_n+2}$, then
 \begin{align}\label{eq:U:n}
 \notag & Q_{n+1,h_n}\stackrel{U_{n-1}}{\rightarrow} Q_{n+1,1}, \\
   & Q_{n+1,2h_n}\stackrel{U_{n-1}}{\rightarrow} Q_{n+1,h_n+1}, \\
  \notag & Q_{n+1,3h_n+1}\stackrel{U_{n-1}}{\rightarrow} Q_{n+1,2h_n+2}, \\
   \notag& Q_{n+1,4h_n+2}\stackrel{U_{n-1}}{\rightarrow} Q_{n+1,3h_n+3}. 
 \end{align}
 \end{itemize} 
 Finally, we consider the adjacent subsquares $Q_{n,h_n}$ and $R_{n-1}$ in Figure 12 and their refinement (equation \ref{series:squares}) on which $V_n$ acts. 
  \begin{figure}[H]
     \centering
     \label{fig:divisions}\includegraphics[scale=0.5]{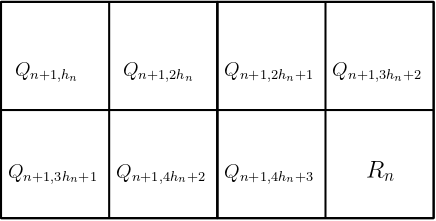}\caption{The adjacent subsquares $Q_{n,h_n}$ and $R_{n-1}$ and their refinement.}
 \end{figure}
\noindent 

\noindent Then the action of $V_n$ is the following
\begin{align*}
& Q_{n+1,h_n}\rightarrow Q_{n+1,2h_n}, \\
& Q_{n+1,2h_n}\rightarrow Q_{n+1,2h_n+1}, \\ 
& Q_{n+1,2h_n+1}\rightarrow Q_{n+1,3h_n+1}, \\
& Q_{n+1,3h_n+1}\rightarrow Q_{n+1,3h_n+2}, \\
& Q_{n+1,3h_n+2}\rightarrow Q_{n+1,4h_n+2}, \\
& Q_{n+1,4h_n+2}\rightarrow Q_{n+1,4h_n+3}, \\
& Q_{n+1,4h_n+3}\rightarrow Q_{n+1,h_n}, \\
& R_n\rightarrow R_n, 
\end{align*}
as described in Figure 13.
\begin{figure}[H]
    \centering
    \includegraphics[scale=0.5]{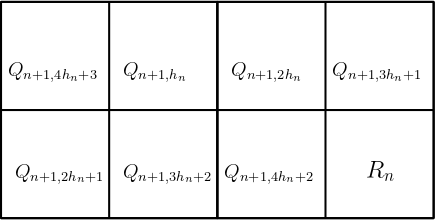}
    \label{fig:action}
    \caption{The action of $V_n$ on the refinement of $Q_{n,h_n}$ and $R_n$.}
\end{figure}
\noindent Thus the map $V_n$ acts only on the sublevels $Q_{n+1,i}$ of the $n+1$-column, not on the sublevels $Q_{n,i}$ of the $n$-th column.
\begin{figure}[H]
    \centering
    \includegraphics[scale=0.5]{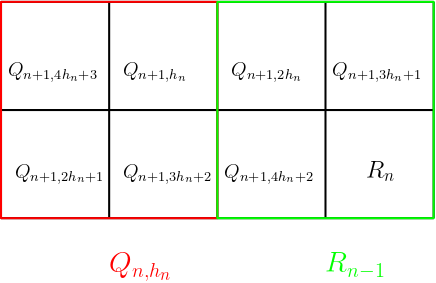}
\end{figure}
\noindent Finally we recall that, by the enumeration chosen,  by using the observation \eqref{obs} and the definition of $U_{n-1}$ in equations \eqref{eq:U:n} one finally has
\begin{equation*}
    Q_{n+1,h_n}\stackrel{V_n}{\longrightarrow}Q_{n+1,2h_n}\stackrel{U_{n-1}}{\longrightarrow}Q_{n+1,h_n+1}=U_{n}(Q_{n+1,h_n}),
\end{equation*}
\begin{equation*}
    Q_{n+1,2h_n}\stackrel{V_n}{\longrightarrow}Q_{n+1,2h_n+1}\stackrel{U_{n-1}}{\longrightarrow}Q_{n+1,2h_n+1}=U_n(Q_{n+1,2h_n}),
\end{equation*}
\begin{equation*}
    Q_{n+1,2h_n+1}\stackrel{V_n}{\longrightarrow}Q_{n+1,3h_n+1}\stackrel{U_{n-1}}{\longrightarrow}Q_{n+1,2h_n+2}=U_n(Q_{n+1,2h_n+1}),
\end{equation*}
\begin{equation*}
    Q_{n+1,3h_n+3}\stackrel{V_n}{\longrightarrow}Q_{n+1,4h_n+2}\stackrel{U_{n-1}}{\longrightarrow}Q_{n+1,3h_n+3}=U_n(Q_{n+1,3h_n+3}),
\end{equation*}
and so on. 
\end{proof}

\noindent We will use similar ideas to recover the flow of the weakly mixing vector field in the next section. 
 \begin{remark}The interested reader can observe that $\forall n$ the map $U_n$ is a cyclic permutation of subsquares of the same area (see \cite{Bianchini_Zizza_residuality} for the definitions). This observation is in the spirit of the result proved by Halmos in \cite{Halmos:weak:mix} that states that weakly mixing automorphisms are a residual $G_\delta$-set in $G(\T^2)$ with the $L^1$-topology. A key ingredient for the proof is to show that cyclic permutations of subsquares are dense in $G(\T^2)$ with the $L^1$-topology. In \cite{Bianchini_Zizza_residuality} similar ideas are used to recover that weakly mixing vector fields are dense with respect to the $L^1_tL^1_x$-topology. In this paper we are doing a different thing: we are considering those automorphisms/vector fields that are weakly mixing but not strongly mixing, fixing a specific cyclic permutation. Therefore we cannot deduce from the previous computations the density of weakly mixing automorphisms. 
 \end{remark}
\section{The weakly mixing vector field}\label{S:wmvf}
In this final section we provide a weakly mixing vector field $b^U\in L^\infty_t\BV_x$ whose RLF $X^U(t)$ evaluated at time $t=1$ is the map $U$ constructed in the previous section. \\

\noindent Our aim is to construct a flow $X^n$ with $n\geq 2$ that, evaluated at time $t=1$, gives the map $V_n$ introduced in \eqref{V:n}. We first consider $\gamma_2\in \mathcal{C}(2)$: we can take to fix the ideas the configuration \eqref{eq:gamma:2}. We start defining the following flow $X^1:[0,1]\times\T^2\rightarrow\T^2$ as
\begin{equation}
    X^1(t)=\begin{cases}
    T(2t)(Q_{2,1},Q_{2,2}),\quad t\in[0,\frac{1}{2}], \\
    T(2t-1)(Q_{2,1},Q_{2,3})\circ T(1)(Q_{2,1},Q_{2,2}),\quad t\in[\frac{1}{2},1].
    \end{cases}
\end{equation}
where the $Q_{2,i}$ $i=1,2,3$ are the subsquares defined in the previous section. One can easily see that $X^1(1)(\gamma_2)=U_1(\gamma_2)$.  \\

\noindent \textbf{Building block flow.} Let us consider the two adjacent subsquares $Q_{n,h_n},R_{n-1}$ of side $\frac{1}{2^n}$ for some $n\in\N$ and consider their subdivisions into the sublevels $Q_{n+1,i}$ as in Proposition \ref{prop:che:palle}.   We look for a flow $X^n(Q_{n,h_n},R_{n-1})[0,1]\times\T^2\rightarrow\T^2$ that moves the subsquares within the time interval $[0,1]$ with the property that $X^n(Q_{n,h_n},R_{n-1})(1)=V_n$. We define therefore $X^n(Q_{n,h_n},R_{n-1})(t)$ as
\begin{equation}
    X^n(Q_{n,h_n},R_{n-1})(t)=\begin{cases}
    X^{S_c}(2^{n+1};2^{n+1}-3,2^{n+1}-1)(7t) \quad t\in[0,\frac{1}{7}],  \\
    \notag \\
    X^{-}(2^{n+1},2^{n+1}-3)(7t-1)\circ X^n(Q_{n,h_n},R_{n-1})(\frac{1}{7}) \quad t\in[\frac{1}{7},\frac{2}{7}], \\
    \notag \\
     X^{S_c}(2^{n+1}-1;2^{n+1}-3,2^{n+1})(7t-2)\circ X^n(Q_{n,h_n},R_{n-1})(\frac{2}{7}) \quad t\in[\frac{2}{7},\frac{3}{7}], \\
     \notag \\
     X^{S_c}(2^{n+1}-1;2^{n+1}-3,2^{n+1})(7t-3)\circ X^n(Q_{n,h_n},R_{n-1})(\frac{3}{7})\quad t\in[\frac{3}{7},\frac{4}{7}], \\
     \notag \\
    X^{+}(2^{n+1},2^{n+1}-3)(7t-4)\circ X^n(Q_{n,h_n},R_{n-1})(\frac{4}{7}) \quad t\in[\frac{4}{7},\frac{5}{7}], \\
    \notag \\
     X^{+}(2^{n+1},2^{n+1}-3)(7t-5)\circ X^n(Q_{n,h_n},R_{n-1})(\frac{5}{7}) \quad t\in[\frac{5}{7},\frac{6}{7}], \\ 
     \notag \\
      X(2^{n+1},2^{n+1}-3; 2^{n+1},2^{n+1}-2)(7t-6)\circ X^n(Q_{n,h_n},R_{n-1})(\frac{6}{7}) \quad t\in[\frac{6}{7},1].
    \end{cases}
\end{equation}
\medskip

\noindent Let us call 
 \begin{equation*}
     b^n(Q_{n,h_n},R_{n-1})(t,x)\doteq\dot{X}^n(Q_{n,h_n},R_{n-1})(t,(X^n)^{-1}(t,x)).
 \end{equation*}
 Then we have the following
\begin{proposition}
 There exist two positive constants $C_1,C_2>0$ such that the following estimates hold:
 \begin{equation}\label{eq:first:estimate}
     \|\dot{X}^n\|_{L^\infty_t L^\infty_x}\leq \frac{C_1}{2^n},
 \end{equation}
 \begin{equation}\label{eq:second:estimate}
     \|TV(b^n)(\T^2)\|_{L^\infty}\leq \frac{C_2}{2^{2n}}.
 \end{equation}
 \end{proposition}
 
 \begin{proof}
 Let us fix $t\in[0,1]$, then by \eqref{eq:est:flow:final}, since the sort operation occurs for a maximum of $4$ adjacent subsquares, that is $j-j'=3$, we have that
 \begin{equation*}
     \|\dot{X}^n\|_{L^\infty_tL^\infty_x}\leq \frac{7\cdot 24}{2^n}=\frac{C_1}{2^n}.
 \end{equation*}
 To prove the second estimate we observe that, by \eqref{eq:est:flow:final:TV}, since again the sort operation occurs between at most $4$ squares, one has
 \begin{equation}
     |\TV(b^n)(t)(\T^2)|\leq 7\cdot\frac{240}{2^{2n}}=\frac{C_2}{2^{2n}},
 \end{equation}
 which concludes the proof.
 \end{proof}

\noindent Finally we can define the flow $X^U(t)$ as 
\begin{equation}
    X^U(t)=\begin{cases}X^1(2t-1) \circ X^U(\frac{1}{2})\quad t\in[\frac{1}{2},1],\\
   
    X^2(Q_{2,3},R_1)(4t-2)\circ X^U(\frac{1}{4}) \quad t\in[\frac{1}{4},\frac{1}{2}], \\
    \dots \\
    X^n(Q_{n,4h_{n}},R_{n-1})(2^nt-2^n+2)\circ X^U(2^{-n}) \quad t\in [2^{-n},2^{-n+1}], \\
    \dots\quad .
    \end{cases}
\end{equation}
One can observe that the map $X^U:[0,1]\times\T^2\rightarrow\T^2$ is well-defined, $X^U\in C([0,1];L^1(\T^2))$ and differentiable $\L^1$-a.e. $t$ and for every $t\in[0,1]$ it is an invertible and measure-preserving map from the torus into itself. We also remark that for every $x\in\T^2$ we have that $\lim_{t\to 0} X^U(t,x)=x$, which tells us that $X^U$ is a flow. \\

\noindent Therefore, by \eqref{vect:field:ass} one can define the divergence-free vector field $b^U(t,x)\doteq \dot{X}^U(t,(X^U)^{-1}(t,x))$.
\begin{proposition}
The divergence-free vector field $b^U$ lives in the space $ L^\infty_t\BV_x$. Moreover, its RLF $X^U(1)$ when evaluated at time $t=1$ is the map $U$, that is $b^U$ is a weakly mixing vector field that is not strongly mixing.
\end{proposition}
\begin{proof}
Let us fix $t\in[2^{-n},2^{-n+1}]$. Then
\begin{equation*}
    b^U(t,x)= 2^n \dot{X}^n(Q_{n,4h_{n-1}+3},R_{n-1})(2^nt-2^n+2,(X^U)^{-1}(t,x)).
\end{equation*}
Therefore
\begin{align}
    \notag \|b^U(t)\|_{L^1_x} &=  2^n \int_{\T^2}|\dot{X}^n(Q_{n,4h_{n-1}+3},R_{n-1})(2^nt-2^n+2,(X^U)^{-1}(t,x))|dx \\ \label{eq:1}
    & =2^n\int_{\T^2}|\dot{X}^n(Q_{n,4h_{n-1}+3},R_{n-1})(2^nt-2^n+2,y)|dy \\ & \label{eq:2} \leq 2^n|Q_{n,4h_{n-1}+3}\cup R_{n-1}| \|\dot{X}^n(Q_{n,4h_{n-1}+3},R_{n-1})\|_{L^\infty_tL^\infty_x} \\ & \label{eq:3}
    \leq 2^n\cdot 2\cdot 2^{-n} \frac{C_1}{2^n}\leq 2C_1,
\end{align}
where \eqref{eq:1} follows by the fact that $X^U(t)$ is measure-preserving, \eqref{eq:2} by the fact that $X^n$ acts only on $Q_{n,4h_n}\cup R_{n-1}$ and \eqref{eq:3} follows by \eqref{eq:first:estimate}. 
To compute the total variation of $b^U$, we observe that, for every $t\in[0,1]$ one has
\begin{equation*}
    \dot{X}^U(t,(X^U)^{-1}(t,x))=2^{n}\dot{X}^n(2^nt-2^n+2,(X^n)^{-1}(2^nt-2^n+2,x)),
\end{equation*}
so that $|\TV(b^U(t)(\T^2)|=2^n|\TV(b^n(t)(\T^2)|$. By using \eqref{eq:second:estimate}, one has that, if $t\in[2^{-n},2^{-n+1}]$
\begin{equation*}
    |\TV(b^U(t)(\T^2)|=2^n|\TV(b^n(t)(\T^2)|\leq 2^n\frac{C_2}{2^{2n}}, 
\end{equation*}
which implies that
\begin{equation*}
    \|\TV(b^U)(\T^2)\|_{L^\infty}\leq C_2.
\end{equation*}
To conclude the proof we have to show that the map $X^U(1)=U$ is the Chacon's map introduced in Section \ref{S:twod}.
More precisely, we have to prove that
\begin{equation*}
    X^1(1)\circ X^2(Q_{2,3},R_1)(1)\circ\dots \circ X^n(Q_{n,4h_{n-1}+3},R_{n-1})(1)=U_n.
\end{equation*}
This easily follows by Proposition \ref{prop:che:palle} observing that the map $X^1(1)=U_1$ and $X^n(Q_{n,4h_{n-1}+3},R_{n-1})(1)=V_n$.
\end{proof}

\printbibliography[title={Bibliography}]
\end{document}